\documentclass[12pt]{article}
\usepackage{amssymb}
\usepackage{amsmath}
\usepackage{amsthm}
\usepackage{authblk}
\usepackage{color}
\usepackage{comment}
\usepackage{geometry}		
\usepackage{graphicx}
\usepackage{url}

\newcommand{\bbbr}{\mathbb{R}}

\begin{document}

\newtheorem{theorem}{Theorem}[section]
\newtheorem{corollary}[theorem]{Corollary}
\newtheorem{lemma}[theorem]{Lemma}
\newtheorem{proposition}[theorem]{Proposition}

\theoremstyle{definition}
\newtheorem{definition}{Definition}[section]
\newtheorem{example}[definition]{Example}

\theoremstyle{remark}
\newtheorem{remark}{Remark}[section]

\title{Differentiable conjugacies for one-dimensional maps}
\author[$\dagger$]{P.A.~Glendinning}
\author[$\ddagger$]{D.J.W.~Simpson}
\affil[$\dagger$]{Department of Mathematics, University of Manchester, Manchester, UK}
\affil[$\ddagger$]{School of Mathematical and Computational Sciences, Massey University, Palmerston North, New Zealand}

\maketitle


\begin{abstract}

Differentiable conjugacies link dynamical systems that share properties such as the stability multipliers
of corresponding orbits. It provides a stronger classification
than topological conjugacy, which only requires qualitative similarity. We describe some of the techniques
and recent results that allow differentiable conjugacies to be defined for standard bifurcations, and explain 
how this leads to a new class of normal forms.
Closed-form expressions for differentiable conjugacies exist between some chaotic maps,
and we describe some of the constraints that make it possible to recognise when such
conjugacies arise. This paper focuses on the consequences
of the existence of differentiable conjugacies
rather than the conjugacy classes themselves.

\end{abstract}

\section{Dynamic conjugacies}

Let $A$ and $B$ be manifolds (in almost everything we do below they are subsets of the real line).
Maps $f:A\to A$ and $g:B\to B$ are conjugate if there exists $h:A\to B$ such that
\begin{equation}\label{eq:conj}
h\circ f = g\circ h.
\end{equation}
The map $h$ is called the conjugating function, and the type of conjugacy depends on properties of $h$. For example 
\begin{itemize}
\item if $h$ is a homeomorphism (continuous bijection with continuous inverse) then $f$ and $g$ are \emph{topologically conjugate};
\item if $h$ is a diffeomorphism (continuously differentiable with continuously differentiable inverse) then $f$ and $g$ are \emph{differentiably conjugate}; and
\item if $h$ is a $C^r$ homeomorphism ($r \ge 1$) then $f$ and $g$ are \emph{$C^r$-conjugate}.
\end{itemize}
(Recall that if a diffeomorphism is $C^r$, meaning its first $r$ derivatives exist and are continuous,
then the first $r$ derivatives of its inverse also exist and are continuous \cite{Bl15}.)
The idea of a conjugacy provides a formal way of saying that different dynamical systems have the `same' dynamics.
It is essentially a change of coordinates: given a homeomorphism $h$ and a map $f$, in the new coordinates $y=h(x)$ we have
\[
y_{n+1}=h(x_{n+1})=h(f(x_n))=h(f(h^{-1}(y_n))).
\]
So $y_{n+1}=g(y_n)$ where $g=h \circ f \circ h^{-1}$, which is an alternative way of writing (\ref{eq:conj}).
Conjugacy classes of one-dimensional maps with fixed points can be studied for their own sake, see \cite{Ofarrell2009,Ofarrell2011} for example, but they are also used
as a tool for solving some larger problem at hand.
In this paper we will concentrate
on the applications of conjugacies to bifurcation theory and chaotic dynamics. We will be particularly interested 
in cases where topological conjugacies can be made differentiable, since differentiable conjugacies preserve many more
features of the dynamics.

To see this suppose that two one-dimensional maps $f$ and $g$ are topologically conjugate by a conjugating function $h$.
Suppose $f$ has a fixed point $x^*$, so $x^*=f(x^*)$, and let $y^*=h(x^*)$.
Then (\ref{eq:conj}) implies
\[
y^*=h(x^*)=h (f(x^*)) = g (h(x^*)) =g(y^*),
\]
so $y^*$ is a fixed point of $g$ --- we say it is
the \emph{corresponding} fixed point of $g$.
Moreover, if $f$ and $g$ are differentiable and $h$ is a diffeomorphism then differentiating (\ref{eq:conj}) and evaluating it at $x^*$ gives
\[
h^\prime (f(x^*))f^\prime (x^*)=g^\prime (h(x^*))h^\prime (x^*).
\]
Since $f(x^*)=x^*$ and $h^\prime (x^*)\ne 0$
(since its inverse is $C^1$), we have
$f^\prime (x^*)=g^\prime (y^*)$. Since the derivative determines stability properties of hyperbolic fixed points (those with the modulus of the derivative not equal to one) and in particular the rates of convergence or divergence of nearby orbits, this means that corresponding fixed points of differentiably conjugate maps have the same local quantitative behaviour as well as qualitative behaviour implied by topological conjugacy. This stability analysis is easily extended to periodic orbits, where the stability of a
period-$p$ orbit $\{ x_1, \ldots, x_p \}$
is determined by the \emph{multiplier},
\begin{equation}\label{eq:mult}
\lambda=\prod_{n=1}^p f^\prime (x_n).
\end{equation} 
Since corresponding periodic orbits have the same multipliers for differentiably conjugate maps, it is not easy to find families of maps arising in applications that are both chaotic and differentiably conjugate. This would require that the multipliers of an infinite number of corresponding periodic orbits are equal, which is an infinite set of constraints. It is of course easy to
reverse engineer such families from
a family of differentiable conjugacies, but in section~\ref{sect:chaotic} we will describe families derived from geometric or algebraic constructions
that are differentiably conjugate (in fact $C^\infty$-conjugate).    

The purpose of this paper is to demonstrate
the application of differentiable conjugacies to two areas of dynamical systems 
theory: bifurcations and chaos. In section~\ref{sect:stern} we review the two main technical 
results that will be required. Sternberg's theorem \cite{Sternberg} provides local smooth conjugacies to linear maps, 
while Belitskii's theorem \cite{Belitskii1986} shows how this can be extended to local basins of attraction and repulsion.
In section~\ref{sect:bifn} we introduce the idea of extended normal forms for bifurcation theory \cite{GS2022,GS2023}. 
These are polynomial extensions of the standard truncated normal forms for local bifurcations, where
the additional terms are chosen so that the extended forms are smoothly conjugate to the original system locally on
basins of attraction and repulsion of fixed points.
This possibility is mentioned in \cite{Ily1993}, but the details were not explored there.
In section~\ref{sect:pws} these results are extended to piecewise-smooth maps, with the initially counter-intuitive 
result that under certain conditions distinct piecewise-smooth maps can be smoothly conjugate \cite{GS2023a}. 
Section~\ref{sect:chaotic} considers results for smooth conjugacies in families of maps,
and uses a remark of Misiurewicz \cite{Mi22} to extend the results of \cite{PGSG2021} to a class of maps studied by Umeno \cite{Umeno1997}, thus 
making a connection between smooth conjugacy and exactly solvable chaos.
Finally section~\ref{sect:concl} provides a short conclusion.

\section{The theorems of Sternberg and Belitskii}
\label{sect:stern}

The technical results needed to address the applications to bifurcation theory and chaotic maps
in later sections are due to Sternberg \cite{Sternberg} and Belitskii \cite{Belitskii1986},
with some more recent results to deal with non-hyperbolic \cite{Ofarrell2011,Young1997} and orientation-reversing \cite{Ofarrell2009} cases.
The differentiable equivalence of hyperbolic fixed points comes from Sternberg \cite{Sternberg}, but we state the result 
in a slightly different form since by restricting to $C^r$ functions with $r \ge 2$ the conjugacy is also $C^r$ \cite{Young1997}
rather than $C^{r-1}$ as in the original statement of \cite{Sternberg}.

\begin{theorem}\label{thm:S}\cite{Sternberg}~Suppose $f:\bbbr \to\bbbr$ is $C^r$ ($r\ge 2$) and $f(0)=0$ with $f^\prime (0)=\lambda$ and $\lambda \notin\{-1,0, 1 \}$.
Then there are open neighbourhoods $U$ of $x=0$ and $V$ of $y=0$ such that $f(x)$ on $U$ is $C^{r}$-conjugate to $g(y) = \lambda y$ on $V$.
\end{theorem}

For $|\lambda|<1$ Sternberg's proof \cite{Sternberg} of this result is based on the analysis of the 
behaviour of $f^n$ as $n \to \infty$ 
(the argument is essentially the same if $|\lambda |>1$ in reverse time).
In \cite{SternbergContract} Sternberg gives an
alternative proof based on an iterative method for the existence of the conjugacy.

Now let $f$ be a map satisfying the conditions of Theorem \ref{thm:S},
and $g$ be another map satisfying the same conditions.
That is, $0$ is a fixed point of both maps and $f'(0) = g'(0) \notin \{ -1, 0, 1 \}$.
Then from two applications of Theorem \ref{thm:S} we
can conclude that $f$ and $g$ are locally $C^r$-conjugate.
Belitskii's theorem \cite{Belitskii1986} shows this
result can be extended to basins of attraction or repulsion of fixed points.
The following lemma indicates the flavour of the general result of Belitskii.

\begin{lemma}Suppose $a<0<b$ and $f:[a,b]\to [a,b]$ is a strictly increasing $C^r$ ($r\ge 2$) map with $f(a)=a$, $f(b)=b$, $f(0)=0$, and $f^\prime (0)=\lambda$ with $0<\lambda <1$.
Further suppose $f$ has no other fixed points on $[a,b]$, so appears as in Fig.~\ref{fig:b}.
Also suppose $\tilde{a} <0 <\tilde{b}$ and $g:[\tilde{a},\tilde{b}]\to [\tilde{a},\tilde{b}]$
is a strictly increasing $C^r$ map with the same properties at corresponding points.
Then $f$ on $(a,b)$ is $C^r$-conjugate to $g$ on $(\tilde{a},\tilde{b})$.
\label{le:Belitskii}
\end{lemma}  

\begin{figure}[b!]
\begin{center}
\includegraphics[height=3.6cm]{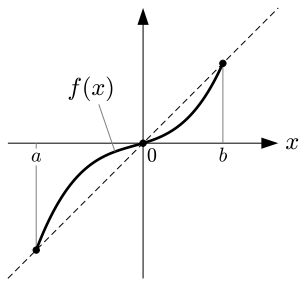}
\caption{
A map satisfying the conditions of Lemma \ref{le:Belitskii}.
\label{fig:b}
} 
\end{center}
\end{figure}

The main idea behind the proof of Lemma \ref{le:Belitskii} is to extend 
Sternberg's conjugacy from a neighbourhood of $x=0$ to the entire interval $(a,b)$,
which is of course the local basin of attraction of the fixed point. Thus we start with $x_0 > 0$ and $y_0 > 0$
such that $[-x_0,x_0]\subset (a,b)$, $[-y_0 ,y_0]\subset (\tilde{a},\tilde{b})$
and $f$ on $[-x_0,x_0]$ is conjugated to $g$ on $[-y_0,y_0]$ by a $C^r$ conjugating function $h$.
That is,
\begin{equation}
h(f(x)) = g(h(x)),
\label{eq:fconj}
\end{equation}
for all $x \in [-x_0,x_0]$.

In order to extend the domain of $h$ towards $b$, we use the backward orbits of $x_0$ and $y_0$ to form sequences
$x_n\to b$ with $f(x_n)=x_{n-1}$
and $y_n\to \tilde{b}$ with $g(y_n)=y_{n-1}$ for all $n \ge 1$.
For any $x \in (x_0,x_1)$, $h$ is defined at $f(x)$,
so we can extend its definition with
\begin{equation}
h(x) = g^{-1}(h(f(x))), \qquad {\rm for~any~} x \in (x_0,x_1).
\nonumber
\end{equation}
By construction the conjugacy relation (\ref{eq:fconj}) holds on the larger domain and is $C^r$ on $(x_0,x_1)$.
To complete the proof it is necessary to show
$h$ is $C^r$ at $x=x_0$ and repeat the construction iteratively to extend $h$ to $[x_n,x_{n+1}]$ for all $n\ge 1$,
and hence to the whole of $[0,b)$.
The argument in $(a,0]$ is similar.
We refer the reader to \cite{Belitskii1986} for details.

\section{Extended Normal Forms}\label{sect:bifn}

Bifurcations are critical parameter values
at which the dynamics of a family of maps undergoes a fundamental (topological) change.
There is a vast theory for bifurcations,
and much of it is based on normal forms \cite{Kuznetsov}.
The basic idea is that a normal form is a family of maps exhibiting the bifurcation
and that can be obtained from any family maps exhibiting the bifurcation through a conjugacy.

For example
$$
g(y,\nu) = y + \nu - y^2,
$$
can be viewed as a normal form for a saddle-node bifurcation because
if an arbitrary family of maps $f(x,\mu)$ has a saddle-node bifurcation,
there exists a homoemorphism $h$ that conjugates it to $g$ locally.
As discussed above, we would of course like $h$ to be differentiable.
However, on the side of the bifurcation where $f$ has two fixed points,
this is only possible if we can match the stability multipliers of both fixed points of $f$
to those of the corresponding fixed points of $g$.
Unless $f$ has a special symmetry this cannot be done because
we cannot tune the single parameter $\nu$ to satisfy both constraints.

However, we can obtain a differentiable conjugacy if we instead
consider the {\em extended normal form}
$$
g(y,\nu,a) = y + \nu - y^2 + a y^3,
$$
which has two parameters, $\nu$ and $a$.
To explain why, suppose $f$ has a saddle-node bifurcation at $(x,\mu) = (0,0)$.
Then
\begin{equation}
f(0,0) = 0, \qquad
\frac{\partial f}{\partial x}(0,0) = 1, \qquad
\frac{\partial f}{\partial \mu}(0,0) > 0, \qquad
\frac{\partial^2 f}{\partial x^2}(0,0) < 0,
\label{eq:saddleNodeAssumptions}
\end{equation}
after substituting $x \mapsto -x$ and/or $\mu \mapsto -\mu$ if necessary to obtain the desired signs.
For small $\mu > 0$, $f$ has two fixed points near $0$, Fig.~\ref{fig:c}.
Via a straight-forward calculation we determine the stability multipliers of these points to be
$$
\lambda^\pm(\mu) = 1 \pm \sqrt{- 2 \textstyle{\frac{\partial f}{\partial \mu} \frac{\partial^2 f}{\partial x^2} \mu}}
- \frac{2 \frac{\partial f}{\partial \mu} \frac{\partial^3 f}{\partial x^3}}{3 \frac{\partial^2 f}{\partial x^2}} \,\mu
+ O \left( \mu^{\frac{3}{2}} \right),
$$
where the derivatives are evaluated at $(x,\mu) = (0,0)$.
Similarly for $\nu > 0$ the map $g$ has two fixed points locally, with stability multipliers
$$
\sigma^\pm(\nu,a) = 1 \pm 2 \sqrt{\nu} + 2 a \nu + O \left( \nu^{\frac{3}{2}} \right).
$$
Thus for $f$ and $g$ to be differentiably conjugate we need
\begin{equation}
\lambda^+(\mu) = \sigma^+(\nu,a), \qquad
\lambda^-(\mu) = \sigma^-(\nu,a).
\label{eq:matchPairsOfMultipliers}
\end{equation}
As shown in \cite{GS2023},
we can use the implicit function theorem to show that
(\ref{eq:matchPairsOfMultipliers}) can indeed be solved for $\nu$ and $a$ locally to obtain the following result.

\begin{figure}[b!]
\begin{center}
\includegraphics[height=3.6cm]{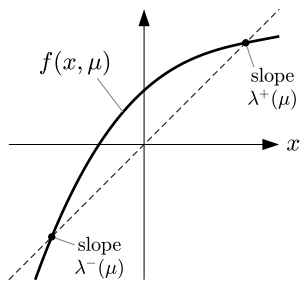}
\caption{
A sketch of a map on one side of a saddle-node bifurcation.
Specifically, $f$ satisfies (\ref{eq:saddleNodeAssumptions}) and is drawn for small $\mu > 0$.
Recall, the stability multiplier of a fixed point of a one-dimensional map
is the slope of the map at the fixed point.
\label{fig:c}
} 
\end{center}
\end{figure}

\begin{theorem}
Suppose $f$ is $C^r$ ($r \ge 4$) and satisfies (\ref{eq:saddleNodeAssumptions}).
Then there exists $\delta > 0$,
neighbourhoods $N$ and $M$ of $0$,
and continuous functions $F, G : (0,\delta) \to \bbbr$ with
\begin{equation}
F(0) = 0, \qquad
G(0) = \frac{2 \frac{\partial^3 f}{\partial x^3}}{3 \left( \frac{\partial^2 f}{\partial x^2} \right)^2} \Bigg|_{(0,0)},
\nonumber
\end{equation}
such that, for all $\mu \in (0,\delta)$, the maps $f(x,\mu)|_N$ and $g(y,F(\mu),G(\mu))|_M$ are $C^{r-1}$-conjugate
on the basins of their corresponding fixed points.
\end{theorem}

Note that in \cite{GS2023} we also show a conjugacy exists for small $\mu \le 0$.
The loss of smoothness 
from $C^r$ to $C^{r-1}$ is due to the fact that the implicit function theorem is applied to a function involving 
the derivatives of $F$ and $G$ which are $C^{r-1}$.

As another example,
$g(y,\nu) = y + \nu y - y^3$
is a normal form for a pitchfork bifurcation,
but again we can usually only obtain a continuous conjugacy.
As shown in \cite{GS2023},
to obtain a differentiable conjugacy it is sufficient to instead use
$g(y,\nu,a,b) = y + \nu y + b \nu y^2 - y^3 + a y^5$.
Here three parameters are needed because on one side of a pitchfork bifurcation there are three fixed points.

The extended normal forms are not unique,
other families of maps can do the same job,
but if we want the extended normal form to be a polynomial with as few terms as possible
the options are rather limited.
For example, in the saddle-node case
we saw that the third derivative of $f$ appears in the $\mu$-coefficient of $\lambda^\pm(\mu)$.
This coefficient and its corresponding one for $g$
are involved in the calculations required to construct $F$ and $G$.
For this reason if we replace $y^3$ in $g$ with a higher power of $y$
we cannot in general solve for $F$ and $G$.

\section{Piecewise-smooth maps}\label{sect:pws}

Piecewise-smooth maps can exhibit a range of bifurcations that are not possible for smooth maps.
In particular, a fixed point can collide with a switching manifold (where the map is non-smooth) giving rise to new dynamics.
Such {\em border-collision bifurcations} have been identified in mathematical models in a wide range of disciplines \cite{Si16}.

The skew tent map family
\begin{equation}
g(y,\nu,s_L,s_R) = \left\{ \begin{array}{l@{~~}l}
\nu + s_L y, & y \le 0, \\
\nu + s_R y, & y \ge 0,
\end{array} \right.
\label{eq:skewTent}
\end{equation}
can in some ways be regarded as a normal form for border-collision bifurcations \cite{DiBu08}.
As the value of $\nu \in \bbbr$ passes through $0$, a border-collision bifurcation occurs
and the resulting dynamics depends in a complicated but well understood way on the values of $s_L, s_R \in \bbbr$
\cite{AvGa19,ItTa79,MaMa93}.

In order to connect a typical piecewise-smooth map to (\ref{eq:skewTent}) via a topological conjugacy
that is valid over an interval of parameter values,
equality of the kneading sequences \cite{BuKe17,MiVi91}
may mean this is only possible if we allow $s_L$ and $s_R$ to vary with $\nu$.
A differentiable conjugacy will usually not exist if there are several periodic solutions (for the reasons discussed above).
But what about in simple cases where fixed points are the only invariant sets --- can we obtain a differentiable conjugacy?
The answer, perhaps surprisingly, since the maps themselves are not differentiable, is yes.
We just require that the ratio of the slopes at the kink is the same for both maps \cite{GS2023a}.

\begin{figure}[b!]
\begin{center}
\includegraphics[height=3.6cm]{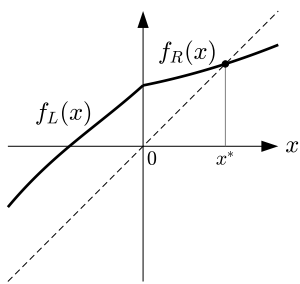}
\caption{
A sketch of a map (\ref{eq:pwsWithoutMu}) discussed in the text.
\label{fig:d}
} 
\end{center}
\end{figure}

To clarify this constraint and show where it comes from, let us derive it with a brief calculation.
Consider a continuous map
\begin{equation}
f(x) = \left\{ \begin{array}{l@{~~}l}
f_L(x), & x \le 0, \\
f_R(x), & x \ge 0,
\end{array} \right.
\label{eq:pwsWithoutMu}
\end{equation}
where $f_L$ and $f_R$ are smooth on $(-\infty,0]$ and $[0,\infty)$
(so derivatives don't blow up at $x=0$).
As indicated in Fig.~\ref{fig:d},
suppose $x^* > 0$ is a fixed point of (\ref{eq:pwsWithoutMu}) with $\lambda = f_R'(x^*) \in (0,1)$
and that $f$ has no other fixed points in an open interval $I$ containing $0$ and $x^*$.
Let 
\begin{equation}
g(y) = \left\{ \begin{array}{l@{~~}l}
g_L(y), & y \le 0, \\
g_R(y), & y \ge 0,
\end{array} \right.
\nonumber
\end{equation}
be another map with the same properties,
i.e.~it has a fixed point $y^* > 0$ with stability multiplier $\lambda$
and no other fixed points in an open interval containing $0$ and $y^*$.

By Sternberg's theorem $f$ and $g$ are differentiably conjugate
on neighbourhoods of $x^*$ and $y^*$.
By Belitskii's lemma (Lemma~\ref{le:Belitskii}) we can extend the conjugacy 
to the left until reaching either $x=0$ or $y=0$.
But in fact we can find a conjugacy $h$ so that $x=0$ and $y=0$ are reached \emph{simultaneously}
by choosing $h$ so that it maps the forward orbit of $x=0$ under $f$ to the forward orbit of $y=0$ under $g$.
That is
\begin{equation}
h(f_R(x)) = g_R(h(x)),
\label{eq:conjugacyR}
\end{equation}
for all $0 \le x \le x^*$, with $h(0) = 0$ and $h(x^*) = y^*$.
Differentiating (\ref{eq:conjugacyR}) gives
\begin{equation}
h'(x) = \frac{h'(f_R(x)) f_R'(x)}{g_R'(h(x))},
\nonumber
\end{equation}
so in particular
\begin{equation}
\lim_{x \to 0^+} h'(x) = \frac{h'(f_R(0)) f_R'(0)}{g_R'(0)}.
\label{eq:rightDeriv}
\end{equation}
By repeating the procedure used in \S\ref{sect:stern},
we extend the domain of $h$ to the left by defining
\begin{equation}
h(x) = g_L^{-1}(h(f_L(x))),
\label{eq:conjugacyL}
\end{equation}
for small $x < 0$.
By construction $h$ provides a conjugacy from $f$ to $g$ on the larger domain
and is differentiable for small $x < 0$ but possibly non-differentiable at $x=0$.
Differentiating (\ref{eq:conjugacyL}) and taking $x \to 0$ from the left gives
\begin{equation}
\lim_{x \to 0^-} h'(x) = \frac{h'(f_L(0)) f_L'(0)}{g_L'(0)}.
\label{eq:leftDeriv}
\end{equation}
By matching (\ref{eq:rightDeriv}) and (\ref{eq:leftDeriv})
and using the fact that $h$ is differentiable at $f_L(0) = f_R(0)$,
we conclude that $h$ is differentiable at $x=0$ if and only if
$\frac{f_L'(0)}{f_R'(0)} = \frac{g_L'(0)}{g_R'(0)}$.
That is, the \emph{slope ratios} at the kinks $x=0$ and $y=0$ are the same for both maps.

We now demonstrate the consequences of this to conjugacies for border-collision bifurcations.
Consider a family of piecewise-smooth maps
\begin{equation}
f(x,\mu) = \left\{ \begin{array}{l@{~~}l}
f_L(x,\mu), & x \le 0, \\
f_R(x,\mu), & x \ge 0.
\end{array} \right.
\label{eq:pws}
\end{equation}
Continuity at $x=0$ implies $f_L(0,\mu) = f_R(0,\mu)$ for all values of $\mu$.
Suppose $x=0$ is a fixed point of (\ref{eq:pws}) with $\mu=0$, i.e.
\begin{equation}
f_L(0,0) = f_R(0,0) = 0.
\label{eq:pwsBCBAssumption}
\end{equation}
Let
\begin{equation}
a_L = \frac{\partial f_L}{\partial x}(0,0), \qquad
a_R = \frac{\partial f_R}{\partial x}(0,0), \qquad
\beta = \frac{\partial f_L}{\partial \mu}(0,0) = \frac{\partial f_R}{\partial \mu}(0,0),
\nonumber
\end{equation}
and suppose
\begin{equation}
a_L > 1, \qquad
0 < a_R < 1, \qquad
\beta > 0.
\label{eq:pwsAssumptions}
\end{equation}
In a neighbourhood of $(x,\mu) = (0,0)$,
for $\mu < 0$ the map has no fixed points,
while for $\mu > 0$ it has two fixed points, Fig.~\ref{fig:e}.
Locally the map is monotone
and as the value of $\mu$ is varied through $0$
the border-collision bifurcation mimics a saddle-node bifurcation.

\begin{figure}[b!]
\begin{center}
\includegraphics[height=3.6cm]{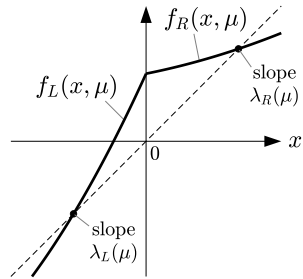}
\caption{
A piecewise-smooth map (\ref{eq:pws}) satisfying (\ref{eq:pwsAssumptions}) for small $\mu > 0$.
\label{fig:e}
} 
\end{center}
\end{figure}

To obtain a differentiable conjugacy between $f$ with $\mu > 0$ and a similar map,
we need to match the stability multipliers of both fixed points and the slope ratio at the kink.
Straight-forward calculations reveal that the fixed points have multipliers
\begin{eqnarray}
\lambda_L(\mu) &=& \textstyle{\frac{\partial f_L}{\partial x}}
+ \left( \textstyle{\frac{\partial^2 f_L}{\partial \mu \partial x}}
+ \frac{\beta \frac{\partial^2 f_L}{\partial x^2}}{1 - \frac{\partial f_L}{\partial x}} \right) \mu + O \left( \mu^2 \right), \nonumber \\
\lambda_R(\mu) &=& \textstyle{\frac{\partial f_R}{\partial x}}
+ \left( \textstyle{\frac{\partial^2 f_R}{\partial \mu \partial x}}
+ \frac{\beta \frac{\partial^2 f_R}{\partial x^2}}{1 - \frac{\partial f_R}{\partial x}} \right) \mu + O \left( \mu^2 \right), \nonumber
\end{eqnarray}
with derivatives evaluated at $(x,\mu) = (0,0)$,
and the slope ratio is
\begin{equation}
S(\mu) = \frac{\frac{\partial f_L}{\partial x}(0,\mu)}{\frac{\partial f_R}{\partial x}(0,\mu)}.
\nonumber
\end{equation}

Thus there are three constraints, and the skew tent map family (\ref{eq:skewTent})
does indeed have three parameters, but tuning the value of $\nu$
does not help us satisfy these constraints because in (\ref{eq:skewTent})
the value of $\nu > 0$ can be scaled to $1$
(since $g(\gamma y, \gamma \nu, s_L, s_R) = \gamma g(y,\nu,s_L,s_R)$ for any $\gamma > 0$).
So instead we can consider
\begin{equation}
g(y,\nu,s_L,s_R,t) = \left\{ \begin{array}{l@{~~}l}
\nu + s_L y + t y^2, & y \le 0, \\
\nu + s_R y, & y \ge 0.
\end{array} \right.
\label{eq:skewSingleQuad}
\end{equation}
If $0 < s_R < 1 < s_L$ then for small $\nu > 0$ the fixed points of $g$ have multipliers
\begin{eqnarray}
\sigma_L(\nu,s_L,t) &=& s_L + \frac{2 t}{1 - s_L} \,\nu + O \left( \nu^2 \right), \nonumber \\
\sigma_R(s_R) &=& s_R \nonumber \,,
\end{eqnarray}
and the slope ratio at $y=0$ is $\frac{s_L}{s_R}$.
Thus $f$ and $g$ are locally conjugate if
\begin{equation}
\lambda_L(\mu) = \sigma_L(\nu,s_L,t), \qquad
\lambda_R(\mu) = \sigma_R(s_R), \qquad
r(\mu) = \frac{s_L}{s_R}.
\nonumber
\end{equation}
It turns out we can solve these to obtain $s_L$, $s_R$, and $t$ as functions of $\mu$,
using also $\nu = \mu$ due to the above scaling property, leading to the following result \cite{GS2023a}.

\begin{theorem}
Let $f$ be a piecewise-$C^3$ map (\ref{eq:pws}) satisfying (\ref{eq:pwsBCBAssumption})--(\ref{eq:pwsAssumptions}),
and $g$ be given by (\ref{eq:skewSingleQuad}).
Then there exists $\delta > 0$, neighbourhoods $N$ and $M$ of $0$,
and continuous functions $F_L, F_R, G : (0,\delta) \to \bbbr$ with
\begin{equation}
F_L(0) = a_L, \qquad
F_R(0) = a_R, \qquad
G(0) = \textstyle{\frac{\beta}{2} \left( \frac{\partial^2 f_L}{\partial x^2}
- \frac{a_L(1-a_L)}{a_R(1-a_R)} \frac{\partial^2 f_R}{\partial x^2} \right) \Big|_{(0,0)}},
\nonumber
\end{equation}
such that $f(x,\mu)|_N$ and $g(y,\mu,F_L(\mu),F_R(\mu),G(\mu))|_M$ are
differentiably conjugate on the basins of their corresponding fixed points for all $\mu \in (0,\delta)$.
\end{theorem}


\section{Chaotic maps}\label{sect:chaotic}

The elliptic curves $y^2=x^3+ax+b$ have many beautiful properties. One of these is that any line tangential to an elliptic curve intersects the curve at one other point (after adding the point at infinity). This process can be iterated: start at a point on the curve, find the point determined by the tangent at the original point, and now find the point determined by the tangent at the new point, and so on. This generates a map by (for example) considering the $x$-coordinates of successive points determined by this process, which gives
\begin{equation}\label{eq:FSS}
F(x)=\frac{x^4-2ax^2-8bx+a^2}{4(x^3+ax+b)},
\end{equation}
see Fig.~\ref{fig:g}.
Equation (\ref{eq:FSS}) is a family of maps parametrized by $a$ and $b$
and it is not hard to show that these are all topologically conjugate to a full shift on two symbols and hence to the standard quadratic map (Chebyshev map)
\begin{equation}\label{eq:T2}
T_2(x)=1-2x^2,  
\end{equation}
on the interval $[-1,1]$, Fig.~\ref{fig:f}.
Glendinning and Glendinning \cite{PGSG2021} show that each of the maps $F$ of (\ref{eq:FSS}), or more accurately a compactification of (\ref{eq:FSS}),
are $C^\infty$-conjugate to each other and to $T_2$. The proof of this statement is based on two observations. First, Jiang \cite{Jiang1995,Jiang1997,Jiang2005} has shown that chaotic unimodal maps which are topologically conjugate, have the same types of turning points (in this case quadratic), and for which corresponding periodic orbits have the same multipliers, are $C^\infty$-conjugate. The first two criteria are easy to establish for the maps defined by (\ref{eq:FSS}) so it is the infinite set of equalities of corresponding multipliers that presents a challenge. This is solved by the following lemma.

\begin{figure}[b!]
\begin{center}
\includegraphics[height=3.6cm]{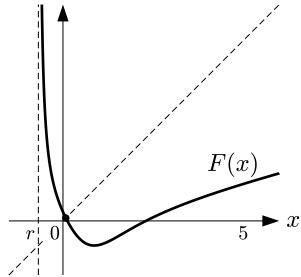}
\caption{
The map (\ref{eq:FSS}) on $(r,\infty)$, where $r$ is the vertical asymptote, derived from elliptic curves with $a = b = 1$.
After compactification by a real M\"obius transformation this is conjugate to the Chebyshev map (\ref{eq:T2}) on $(-1,1)$.
\label{fig:g}
} 
\end{center}
\end{figure}

\begin{figure}[b!]
\begin{center}
\includegraphics[height=3.6cm]{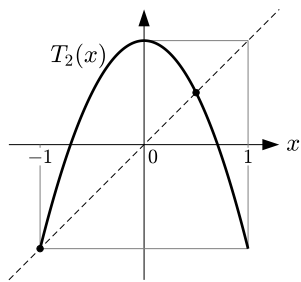}
\caption{
The Chebyshev map (\ref{eq:T2}).
\label{fig:f}
} 
\end{center}
\end{figure}

\begin{lemma}\label{lem:PGSG}\cite{PGSG2021}~Let $F$ be defined by (\ref{eq:FSS}) and $H(x)=x^3+ax+b$. Then
\begin{equation}\label{eq:ratF}
H(F(x))=\frac{1}{4}[ F^\prime (x)]^2H(x) .
\end{equation}\end{lemma}
      
The proof is by brute-force calculation, also verified using the symbolic manipulation packages of {\tt Mathematica} \cite{Wolfram}.
This has the immediate consequence that if $\{x_1, \dots x_p\}$ is a period-$p$ orbit of $F$ and $H(x_n)\ne 0$ for each $n$, 
then the stability multiplier of the periodic orbit is $\lambda$ with 
\begin{equation}\label{eq:multF}
\lambda^2=\prod_{n=1}^p [ F^\prime (x_n)]^2=4^p\frac{H(x_1)}{H(x_p)}\prod_{n=1}^{p-1}\frac{H(x_{n+1})}{H(x_n)}=4^p.
\end{equation}
In other words, provided $H \ne 0$ at periodic points,
the modulus of the multiplier of every period-$p$ orbit
is $2^p$, and this is independent of the parameters $a$ and $b$.
There are special points which need to be checked by hand:
in this case the endpoints where the derivative at the fixed point is $4$ and not $2^1$.

The usual proof that the same is true for $T_2$ uses the conjugacy to the tent map with slopes $\pm 2$ \cite{Devaney},
but here let us demonstrate this using the same idea as Lemma~\ref{lem:PGSG}. 

\begin{lemma}\label{lem:quad}Let $H_2(x)=1-x^2$. Then
\begin{equation}\label{eq:ratT2}
H_2(T_2(x))=\frac{1}{4}[ T_2^\prime (x)]^2H_2(x).
\end{equation}\end{lemma}

\begin{proof}
By straightforward calculation
\[
H_2(T_2(x))=1-(1-2x^2)^2=4x^2-4x^4,
\]
and since $T_2^\prime (x)=-4x$,
\[
\frac{1}{4}[ T_2^\prime (x)]^2H_2(x)=4x^2(1-x^2).
\]
\emph{End of proof.}
\end{proof}

Note that at $x=\pm 1$, $H(x)=0$ and so the argument using (\ref{eq:multF}) does not hold.
As before at the fixed point $x=-1$ the derivative is different: again it is $4$ as can be checked by hand.
In the case of the rescaled quadratic map $\tilde{T}(x)=4x(1-x)$ on $[0,1]$ the equivalent $H$ function is $\tilde{H}(x)=x(1-x)$ and (\ref{eq:ratT2}) holds with $(T_2,H_2)$ replaced by $(\tilde{F},\tilde{H})$. This will be useful below.

Lemmas~\ref{lem:PGSG} and \ref{lem:quad} suggest a much stronger principle at work. Misiurewicz \cite{Mi22} has pointed out that if a functional relationship of the form
\begin{equation}\label{eq:Mi}
H(F(x))={\mathcal G}(F^\prime(x))H(x)
\end{equation}
holds for a family of topologically conjugate maps $F$ then the equal multipliers at corresponding periodic orbits 
condition holds. This shows that the relationship (\ref{eq:ratF}) of \cite{PGSG2021} should have been no surprise. Armed with this insight it was possible to find (\ref{eq:ratT2}) with ease, and it provides a method for determining whether other examples have similar properties. 

As pointed out in \cite{PGSG2021}, the examples of \emph{exactly solvable} maps of Umeno \cite{Umeno1997,Umeno1999} bear strong similarities to the elliptic curve example (\ref{eq:FSS}). These maps are constructed using hypergeometric function theory to have ergodic measures that can be written down explicitly. A first example \cite{Umeno1997} is the iteration of the one parameter family of Katsura-Fukuda maps, $F_{\ell}:[0,1]\to [0,1]$ defined by 
\begin{equation}\label{eq:U1}
F_{\ell}(x)=\frac{4x(1-x)(1-\ell x)}{(1-\ell x^2)^2},
\end{equation}
with $\ell\in [0,1)$, Fig.~\ref{fig:h}. If $\ell =0$ this is the rescaled quadratic map $\tilde{T}(x)$
and the maps are constructed so that the invariant measure has density
\begin{equation}\label{eq:U1dens}
\rho_{\ell}=\frac{1}{2K(\ell )\sqrt{ x(1-x)(1-\ell x)}},
\end{equation}
where the normalization constant is the elliptic integral of the first kind \cite{Umeno1997}.

\begin{figure}[b!]
\begin{center}
\includegraphics[height=3.6cm]{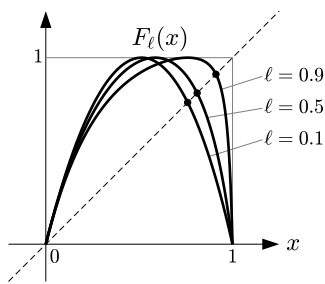}
\caption{
The Katsura-Fukuda map (\ref{eq:U1}) for three different values of $\ell$.
\label{fig:h}
} 
\end{center}
\end{figure}

In this case it is possible to go through the same analysis as \cite{PGSG2021}, though the functions are more complicated than those of Lemma~\ref{lem:quad} and were discovered using {\tt Mathematica} \cite{Wolfram}. 

\begin{lemma}\label{lem:U1}Let $F_{\ell}$ be defined by (\ref{eq:U1}) and let $H_{\ell}(x)=x(1-x)(1-\ell x)$ then
\begin{equation}\label{eq:ratU1}
H_{\ell}(F_{\ell}(x))=\frac{1}{4}[ F_{\ell}^\prime (x)]^2H_{\ell}(x) .
\end{equation}\end{lemma}

If $\ell =0$ we recover Lemma~\ref{lem:quad}. Noting that $F_{\ell}^\prime (0) =4$ to deal with the case $H_\ell (0)=0$ we have the equivalent result to that of \cite{PGSG2021} for the family $F_\ell$. 

\begin{theorem}\label{thm:equivU1}For all $\ell, \ell^\prime \in [0,1)$, $F_{\ell}:[0,1]\to [0,1]$ is $C^\infty$-conjugate to $F_{\ell^\prime}:[0,1]\to [0,1]$. \end{theorem}

The proof follows the proof in \cite{PGSG2021} with Lemma~\ref{lem:PGSG} replaced by Lemma~\ref{lem:U1}. And of course this means that the Katsura-Fukuda maps are $C^\infty$-conjugate to $T_2$ and to $F$ (\ref{eq:FSS}) for every $(a,b)\in \bbbr^2$.

\section{Conclusion}\label{sect:concl}

In applications differentiable conjugacies have mostly been used in the following three ways. First, as initial scaling or translations, 
for example to non-dimensionalize a problem or to simplify the parameterization. Second, as part of a 
linearization process, making it possible to deduce details of local behaviour, which may later also be used 
in the study of other phenomena such as global bifurcations, e.g. \cite{Wiggins}. Third,  in the identification 
of the important nonlinear terms near a non-hyperbolic fixed point as a precursor to a deformation argument 
to capture local behaviour at a bifurcation point e.g. \cite{G&H,Kuznetsov}. 

In this paper we have shown that differentiable conjugacies have broader applications.
Bifurcations theorems in most textbooks describe the local dynamics away from the bifurcation 
point via \emph{topological} conjugacy \cite{G&H,Kuznetsov,Wiggins}. This is because of the difficulty
presented by having more than one fixed point locally. The analysis of 
\cite{GS2022,GS2023,Ily1993} shows that this can be made into a differentiable conjugacy on
basins of attraction and repulsion of the fixed points. Moreover, as shown explicitly in \cite{GS2022,GS2023},
the model equations (\emph{extended normal forms}) involve the addition of one or two extra terms 
whose coefficients are appropriate functions of the bifurcation parameter to
the standard truncated normal forms, see section~\ref{sect:bifn}. This idea can be extended to bifurcations in
piecewise-smooth systems as shown in section~\ref{sect:pws}.

In section~\ref{sect:chaotic} we also showed that differentiable conjugacies could be used to show a closer
relationship between some specially constructed examples. This allowed us to reveal commonality between maps
based on elliptic curves \cite{PGSG2021} and the exactly solvable chaotic maps of Umeno \cite{Umeno1997,Umeno1999}. 
These results suggest that further work on differentiable conjugacies
may produce new connections between systems for which only topological conjugacy has been established hitherto.

\section*{Acknowledgements}
\setcounter{equation}{0}

The authors were supported by Marsden Fund contract MAU1809,
managed by Royal Society Te Ap\={a}rangi.

%

\end{document}